\documentclass{amsart}
\usepackage[procnames]{listings}
\usepackage{pdfpages}
\usepackage[procnames]{listings}
\usepackage{color}
\usepackage{hyperref}
\numberwithin{equation}{section}
\usepackage{cite}
\usepackage{relsize}
\usepackage{amssymb}
\numberwithin{equation}{section}
\usepackage{amsmath}
\usepackage{mathtools}
\usepackage{amssymb}
\usepackage{tikz}
\usepackage{tikz-cd}
\usepackage{float}
\usetikzlibrary{matrix,arrows,decorations.pathmorphing}
\title{ The diophantine exponent of the $\mathbb{Z}/q\mathbb{Z}$ points of  $S^{d-2}\subset S^d$}
\author{M. W. Hassan, Y. Mao, N. T. Sardari, R. Smith, X. Zhu}

\address{Department of Mathematics, UW-Madison, Madison, WI 53706}
\email{mwhassan@wisc.edu}

\address{Department of Mathematics, UW-Madison, Madison, WI 53706}
\email{mao36@wisc.edu}

\address{Department of Mathematics, UW-Madison, Madison, WI 53706}
\email{ntalebiz@math.wisc.edu}

\address{Department of Mathematics, UW-Madison, Madison, WI 53706}
\email{rlsmithjr1234@gmail.com}

\address{Department of Mathematics, UW-Madison, Madison, WI 53706}
\email{xzhu274@wisc.edu}

\date{\today}
\usepackage{amsmath}

\usetikzlibrary{positioning}
\renewcommand{\vec}[1]{\mathbf{#1}}

	\newtheorem{thm}{Theorem}[section]
	
	\newtheorem{prop}[thm]{Proposition}
	
        \newtheorem{conj}[thm]{Conjecture}

	\newtheorem{rem}[thm]{Remark}	
	\newtheorem{lem}[thm]{Lemma}

	\theoremstyle{defi}
	\newtheorem{defi}[thm]{Definition}

	\theoremstyle{pf}

	\numberwithin{equation}{section}

\begin{document}
\maketitle
\begin{abstract}\label{abstract}
Assume a polynomial-time algorithm for factoring integers, Conjecture~\ref{conj},  $d\geq 3,$ and $q$ and $p$ are prime numbers, where $p\leq q^A$ for some $A>0$. We develop a polynomial-time algorithm in $\log(q)$ that lifts every $\mathbb{Z}/q\mathbb{Z}$ point of $S^{d-2}\subset S^{d}$ to a $\mathbb{Z}[1/p]$  point of $S^d$ with the minimum height. We implement our algorithm for $d=3 \text{ and }4$. Based on our numerical results, we formulate a conjecture which can be checked in polynomial-time and gives the optimal bound on the  diophantine exponent of the $\mathbb{Z}/q\mathbb{Z}$ points of $S^{d-2}\subset S^d$. 
\end{abstract}
\section{Introduction}\label{introduction}
\subsection{Motivation}
Let $$S^d(R):=\{(x_0,\dots,x_d): x_0^2+\dots+x_d^2=1, \text{ where }x_i\in R \text{ for }  0\leq i\leq d \},$$
where $R$ is any commutative ring. Let $S^{d-2}(R)\subset S^{d}(R)$, be the subset of the points with the coordinates $(x_0,\dots,x_{d-2},0,0)\in S^{d}(R).$
Suppose that \(q \) is a   prime number, and 
\(\vec{a} \in S^d(\mathbb \mathbb{Z}/q\mathbb{Z})\).
Let $p$ be an odd prime number, we say that
\(\vec{s}\in S^d(\mathbb{Z}[1/p])  \) is an integral lift of \(\vec{a}\) if $\vec{s}\equiv \vec{a}~\mathrm{ mod }~q.$ Let $H:S^d(\mathbb{Z}[1/p]) \to \mathbb{Z}^{+}$ be the natural height function defined by  
$$H((\frac{n_0}{p^{h_0}},\dots,\frac{n_d}{p^{h_d}})):=\max_{0\leq i\leq d}{p^{h_i}}.$$
%\(\sum s_i^2 = p^{2h}\), for \(p\) an odd prime distinct from \(q\) and
%\(h\) a natural number, so that \(s/p^h\) is a point on the unit sphere
%and has coordinates in \(\mathbb Z[1/p]\). 
where $\gcd(n_i,p)=1$ for $0\leq i\leq d.$ We define the diophantine exponent of \(\vec{a} \in S^d(\mathbb \mathbb{Z}/q\mathbb{Z})\) with respect to $p$ to be 
$$
w_p(\vec{a}):=\frac{d-1}{d}\min\log_q(H(\vec{s})): \vec{s}\in S^d(\mathbb{Z}[1/p]) \text{ lifts } \vec{a} .
$$
Assume that $d\geq 3.$ By the circle method (Hardy-Littlewood circle method for $d\geq 4$ and its refinement by Kloosterman \cite{Kloos} for $d=3$), it follows that $w_p(\vec{a})<\infty$ for every $\vec{a}\in S^d(\mathbb \mathbb{Z}/q\mathbb{Z})\). Moreover,  it follows from the circle method  that the number of the integral points $\vec{s}\in S^d(\mathbb{Z}[1/p])$ with $H(\vec{s})\leq p^{h}$ is less than  $O_{\epsilon}( p^{h(d-1+\epsilon)})$ for any $\epsilon>0$. It is elementary to check that $\#S^d(\mathbb \mathbb{Z}/q\mathbb{Z}) \gg q^d.$ Hence,   by a Pigeon-hole argument, it follows that  $w_p(\vec{a})\geq 1+o_q(1)$ for all but a tiny fractions of 
$\vec{a}\in S^d(\mathbb \mathbb{Z}/q\mathbb{Z}).$ It follows from the work of the third author \cite[Theorem 1.2]{optimal} that $w_p(\vec{a})\leq 2-2/d+o_{q,\epsilon}(1)$ for every \(\vec{a}\in S^d(\mathbb \mathbb{Z}/q\mathbb{Z})\), $d\geq 4$, and $p\leq q^{\epsilon},$ where $o_{q,\epsilon}(1)\to 0$ as $q\to \infty$ and $\epsilon\to 0.$ Moreover, this bound is essentially optimal. The third author also conjectured \cite[Conjecture 1.3]{optimal} that $w_p(\vec{a})\leq 4/3+o_{q,\epsilon}(1)$ for $d=3.$ This is the non-archimedian version of  Sarnak's conjecture on the covering exponent of integral points on the sphere; see \cite{SarnakGate}, \cite{optimal}, and \cite{BKS}. 

The main motivation for studying  $w_p(\vec{a})$ for  \(\vec{a} \in S^{d-2}(\mathbb \mathbb{Z}/q\mathbb{Z})\subset S^d(\mathbb \mathbb{Z}/q\mathbb{Z})\)  comes from the navigation algorithms for the  LPS Ramanujan graphs $X^{p,q},$ and its archimedean analogue which is the Ross and  Selinger algorithm for navigating $PSU(2)$ with the golden quantum gates; see \cite{LPS},  \cite{Margulis}, \cite{Petit2008}, \cite{complexity}, and also \cite{Selinger1} and \cite{Parzanchevski}. 

More precisely, the vertices of the LPS Ramanujan graph $X^{p,q}$ are labeled with $\pm\vec{a} \in S^3(\mathbb \mathbb{Z}/q\mathbb{Z})/\pm$, if $p$ is a quadratic residue mod $q.$ It follows from \cite{LPS} and \cite[Theorem 1.7]{complexity} that   the shortest path between  $\pm\vec{a}$ and $\pm(1,0,\dots,0)$ with even number of edges is $\log_p(q^3) w_p(\vec{a})$. In \cite[Theorem 1.2]{complexity}, the third author developed and implemented a conditional polynomial-time algorithm that gives the shortest possible path between any  \( \pm \vec{a} \in S^{1}(\mathbb \mathbb{Z}/q\mathbb{Z})\subset S^3(\mathbb \mathbb{Z}/q\mathbb{Z})\) and $\pm(1,0,\dots,0).$ He also proved that finding the shortest possible path between a generic point  \( \pm \vec{a} \in S^3(\mathbb \mathbb{Z}/q\mathbb{Z})\) and $\pm(1,0,\dots,0)$ is essentially NP complete \cite[Corollary 1.9]{complexity}. The archimedean analogue of this NP-completeness result is in the work of Sarnak and Parzanchevski \cite{Parzanchevski}.

Therefore, the diophantine exponent $w_p(\vec{a})$ for \(\vec{a} \in S^{1}(\mathbb \mathbb{Z}/q\mathbb{Z})\subset S^3(\mathbb \mathbb{Z}/q\mathbb{Z})\) and its archimedean analogue is proportional to the size of the output of  these navigation algorithms. Understanding  the size of the output of these algorithms  is a fundamental problem in quantum computing. Since it helps us to optimize the cost of the implementation of an algorithm on a quantum computer  if one is ever build.  

The only known upper bound  is $w_p(\vec{a})\leq 2+o(1)$ \cite{LPS} which implies $\text{diam}X^{p,q} \leq (2+o(1))\log_p(q^3)$, where $\text{diam}X^{p,q}$ is the diameter of the LPS Ramanujan graph $X^{p,q}$.   The third author also proved that  $\text{diam}(X^{p,q}) \geq (4/3) \log_p(q^3)$ for some integral values of $q$\cite{Naser}.  The third author conjectured that $\text{diam}(X^{p,q})=4/3 \log_p(q^3)+o_q(1).$

The numerical results of Ross and Selinger \cite{Selinger1} and the third author \cite{complexity,Naser} suggests that for all but tiny fractions of \(\vec{a} \in S^{1}(\mathbb \mathbb{Z}/q\mathbb{Z})\subset S^3(\mathbb \mathbb{Z}/q\mathbb{Z})\), we have $w_p(\vec{a})=1+o_q(1).$ It is also observed that $\max_{\vec{a}}(w_p(\vec{a}))=4/3+o_q(1).$ The main goal of this paper is to give a theoretical explanation to these observation. 

\subsection{Main results}
In this paper, we develop a conditional  polynomial-time algorithm for lifting every \(\vec{a} \in S^{d-2}(\mathbb \mathbb{Z}/q\mathbb{Z})\subset S^d(\mathbb \mathbb{Z}/q\mathbb{Z})\)  to an integral point \(\vec{s}\in S^d(\mathbb{Z}[1/p])  \) with the minimal height. In particular, we have a conditional polynomial time algorithm in $\log(q)$  that computes $w_p(\vec{a})$ for every \(\vec{a} \in S^{d-2}(\mathbb \mathbb{Z}/q\mathbb{Z})\subset S^d(\mathbb \mathbb{Z}/q\mathbb{Z})\). We prove that our algorithm terminates in polynomial-time by assuming  a polynomial-time algorithm for factoring integers and  an arithmetic conjecture, which we formulate next.

 Let $\vec{t}:=(t_0,\dots,t_{d-2})$ and $Q(\vec{t}):=\frac{N}{q^2}-(t_0+\frac{b_0}{q})^2-\dots-(t_{d-2}+\frac{b_{d-2}}{q})^2$, where $N$, $b_0, \dots, b_{d-2}$ are integers,   $N\equiv b_0^2+\dots+b_{d-2}^2 \mod q,$ and $\gcd(N,q)=1.$ Define 
\begin{equation}\label{Afr}
A_{Q,r}:=\big\{\vec{t}\in\mathbb{Z}^{d-1}: Q(\vec{t})\in \mathbb{Z}, |\vec{t}|<r, \text{ and } Q(\vec{t})\geq 0  \big\},
\end{equation}
where $r>0$ is some positive real number. 
 
% 
%  be an integral polynomial of degree 2 with $n\geq 3$ variables. Let $V_Q$ be the quadric associated to $Q(x_0,\dots,x_{n-1})=0.$ Define the $p$-adic density of $V_Q$ by:
%$$
% \sigma_p(V_Q):=\lim_{k\to \infty }\frac{ \# \{(x_0,\dots,x_{n-1})\in \big(\frac{\mathbb{Z}}{p^{h^{\prime}}\mathbb{Z}}\big)^n: Q(x_0,\dots,x_{n-1})=0\mod p^{h^{\prime}}  \}}{p^{k(d-1)}}.
%$$
%It follows from  Hensel's lemma that the above limit exists.

 % Conjecture~(\ref{cc}). Conjecture~(\ref{cc}) is a Cramer type conjecture about the distribution of numbers representable as a sum of two squares. 
 
\begin{conj}\label{conj}
%Let $F(x_1,\dots,x_s)$ be a degree 2 integral polynomial in $s \geq 1$ variables, where the absolute value of all the coefficients are  less than $M$ for some $M>0$. Let $Q(x_1,\dots,x_{s+2}):=F(x_1,\dots,x_s)-x_{s+1}^2-x_{s+2}^2.$  Suppose  that $\prod_p\sigma_p(V_Q) \geq c_0$  where  $c_0>0$ is a  fixed constant independent  of $M$, e.g. $c_0=10^{-2}$. Moreover, assume that $F$ is non-negative on some box $I_1\times \dots \times I_s\subset \mathbb{R}^s$, where  $I_j= [-b_j,b_j]$ are some intervals. Then, 
Let $Q$ and $A_{Q,r}$ be as above. There exists constants $\gamma>0$ and $C_{\gamma}>0$, independent of $Q$ and $r$, such that if $ |A_{Q,r}|> C_{\gamma}(\log N)^{\gamma}$ for some $r>0$, then $Q$ expresses a sum of two squares inside $ A_{Q,r}$. \end{conj}
%Recently, Wooley~\cite{Wooley} established an average version of the above conjecture in some special cases.  

 We denote the following  assumptions by $(*)$:
  \begin{enumerate}\label{assumptions}
 \item  There exists a  polynomial-time algorithm for factoring integers,
 \item  Conjecture~\ref{conj} holds.
 \end{enumerate}
This is a version of our  main theorem. 
\begin{thm} \label{main}
Assume~$(*),$ $d\geq 3$ is fixed, $p\neq q$ and $p\leq q^{A}$ for some fixed $A>0$. We develop a  deterministic  polynomial-time algorithm  in $\log(q)$,   that on input $\vec{a}\in S^{d-2}(\mathbb \mathbb{Z}/q\mathbb{Z})\subset S^d(\mathbb \mathbb{Z}/q\mathbb{Z})$  returns a minimal  lift  \(\vec{s}\in S^d(\mathbb{Z}[1/p])  \) of $\vec{a}$.
\end{thm}
 We prove Theorem~\ref{main} in Section~\ref{mainthm}.
 \begin{rem}
 By \cite[Corollary 1.9]{complexity},  finding  a minimal lift of a generic point $\vec{a}\in S^d(\mathbb \mathbb{Z}/q\mathbb{Z})$ is essentially NP-complete. Moreover,  Theorem~\ref{main} generalizes the lifting algorithm  for the $\mathbb{Z}/q\mathbb{Z}$ points of $S^1\subset S^3$ \cite[Theorem 1.10]{complexity}  to $S^{d-2}\subset S^d$ for any $d\geq 3.$ 
 \end{rem}
 The main observation of this paper links  $w_p(\vec{a})$ to another invariant associated to $\vec{a}\in S^{d-2}(\mathbb \mathbb{Z}/q\mathbb{Z})\subset S^d(\mathbb \mathbb{Z}/q\mathbb{Z}),$ which we describe next. Suppose that $\vec{a}=(a_0,\dots,a_{d-2},0,0).$ Let $\mathcal L(\vec{a})$ be the following sub-lattice of $\mathbb{Z}^{d-1}$ with co-volume $q$:
\begin{equation}\label{L(a)}
\mathcal L(\vec{a}):=\big\{(x_0,\dots,x_{d-2})\in \mathbb{Z}^{d-1}:x_0a_0+\dots+x_{d-2}a_{d-2}\equiv 0~\mathrm{ mod }~q   \big\}.
\end{equation}
For any $\mathbb{Z}$ basis \(B = \{\vec v_1,\dots,\vec v_{d-1}\}\) of \(\mathcal L(\vec{a})\),
let 
\begin{equation}\label{MB}
M(B) = \max\{|\vec v_1|,\dots,|\vec v_{d-1}|\},
\end{equation} where \(|\vec{v}|\) is the Euclidean norm of  $\vec{v}$. Define the height function 
$$
\eta(\vec{a}):=\log_q \min_{B} M(B),
$$
where $B$ varies among all $\mathbb{Z}$ basis of \(\mathcal L(\vec{a})\). We prove that $\eta(\vec{a})$ is computable in polynomial-time in $\log(q)$ up to an error term of size $O_d(1/\log(q)).$
\begin{thm}\label{algorithm}
Fix $d\geq 3.$ We develop  a  deterministic  polynomial-time algorithm  in $\log(q)$,   that on input $\vec{a}\in S^{d-2}(\mathbb \mathbb{Z}/q\mathbb{Z})\subset S^d(\mathbb \mathbb{Z}/q\mathbb{Z})$  returns $\eta(\vec{a})+O_d(1/\log(q))$.
\end{thm}
   
We implemented  the algorithms in Theorem~\ref{main} and Theorem~\ref{algorithm} for $d=4$ \cite{gitlab}.  Figure~\ref{linear}  illustrates our main observation, which links the diophantine exponent $w_p(\vec{a})$ to the height function $\eta(\vec{a})$.  
%This can be viewed as a function on
%the projective space \(P^2(\mathbb Z/q\mathbb Z)\), since \(\mathcal L\)
%is invariant under multiplication by non-zero scalars. 
%
\begin{figure}[H]
\centering
\raisebox{-0.6cm}{
\includegraphics[width=\textwidth]{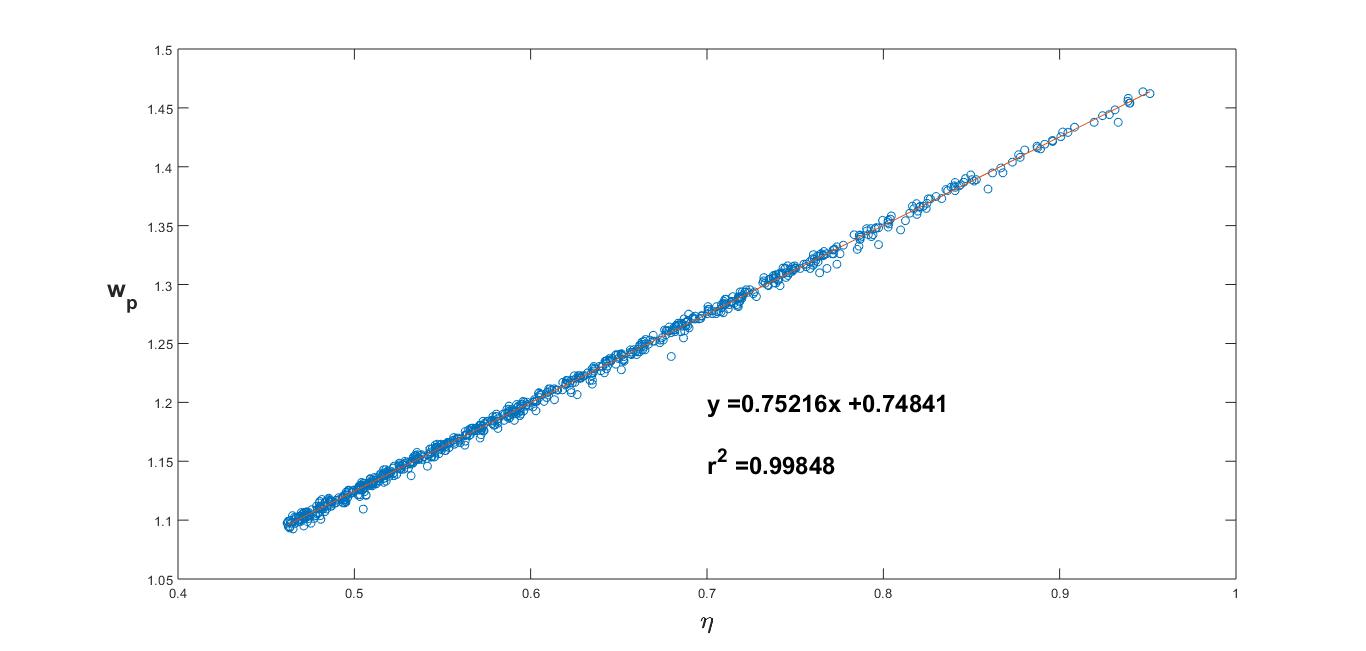}
}
\caption{Random Coordinates}
\label{linear}
\end{figure}
We graph $w_p(\vec{a})$ against  $\eta(\vec{a})$ for $\vec{a}$ chosen randomly on a logarithmic scale and eight $130$-digit values of $q$, as described in Section~\ref{numeric}.  Figure~\ref{linear} suggests the following linear relation between $w_p(\vec{a})$ and $\eta(\vec{a})$ 
\begin{equation}\label{expiden}
w_p(\vec{a}) = \frac{3}{4}(1 + \eta(\vec{a})) + o_q(1).
\end{equation}
We give further numerical evidences that supports the above relation in Section~\ref{numeric}.
Moreover, we prove the following theorem in Section~\ref{mainthm}.
\begin{thm}\label{dioexpthm}
Assume $d\geq 3,$ Conjecture~\ref{conj}, and $p\leq \log(q)$. We have
\[
w_p(\vec{a}) \leq \frac{d-1}{d}(1 + \eta(\vec{a})) + O(\log\log(q)/\log(q)),\]
where the implicit constant in $O(\log\log(q)/\log(q))$ only depends on $\gamma$ and $C_{\gamma}$ defined  in Conjecture~\ref{conj} and it is independent of $p$, $q,$ and $\vec{a}.$ 
\end{thm}
Based on our numerical results and Theorem~\ref{dioexpthm}, we conjecture the following optimal upper bound on $w_p(\vec{a})$ for every $\vec{a}\in S^{d-2}(\mathbb{Z}/q\mathbb{Z})\subset S^d(\mathbb{Z}/q\mathbb{Z}).$
\begin{conj}\label{mainconj}
Let $\vec{a}\in S^{d-2}(\mathbb{Z}/q\mathbb{Z})\subset S^d(\mathbb{Z}/q\mathbb{Z})$, $p\leq \log(q)$ be a prime number and $d\geq 3$. We have 
\begin{equation}\label{optexp}
w_p(\vec{a}) \leq \frac{d-1}{d}(1 + \eta(\vec{a})) + O(\log\log(q)/\log(q)),
\end{equation}
where the implicit constant in $O(\log\log(q)/\log(q))$ only depends on $d$  and it is independent of $p$, $q,$ and $\vec{a}.$
\end{conj} 
\begin{rem}
By Theorem~\ref{main} and Theorem~\ref{algorithm}, $w_p(\vec{a})$ and $\eta(\vec{a})$ are computable in polynomial-time in $\log(q).$ Our algorithm for $d=4$~\cite{gitlab} has been implemented and it runs and terminates quickly for  $q\sim 10^{100}.$ We verify Conjecture~\ref{mainconj} for various values of $q$ and $\vec{a}$ in Section 3.    

\end{rem}
We expect that the upper bound~\eqref{optexp} to be sharp for a generic $\vec{a}\in S^{d-2}(\mathbb{Z}/q\mathbb{Z})\subset S^d(\mathbb{Z}/q\mathbb{Z})$ and prime $p\leq \log(q).$ More precisely, we expect that 
\begin{equation}\label{idenmain}
w_p(\vec{a}) = \frac{d-1}{d}(1 + \eta(\vec{a})) + O(\log\log(q)/\log(q))\end{equation}
for fixed $\vec{a}\in S^{d-2}(\mathbb \mathbb{Z}/q\mathbb{Z})\subset S^d(\mathbb \mathbb{Z}/q\mathbb{Z}),$ and all but tiny fractions of primes $1 \leq p \leq \log(q).$ Moreover, by the equidistribution of covolume-1 lattices $\frac{1}{q^{1/(d-1)}}\mathcal L(\vec{a})$ in the space of the unimodular lattices, for  all but a tiny fractions of $\vec{a}\in S^2(\mathbb \mathbb{Z}/q\mathbb{Z})\subset S^4(\mathbb \mathbb{Z}/q\mathbb{Z}),$ we have $\eta(\vec{a})=1/(d-1)+O(\log\log(q)/\log(q)).$ It is also conjectured for $d\geq 3$ that $w_p(\vec{a})=1+O(\log\log(q)/\log(q))$ for all but a tiny fractions of $\vec{a}\in S^{d-2}(\mathbb \mathbb{Z}/q\mathbb{Z})\subset S^d(\mathbb \mathbb{Z}/q\mathbb{Z}).$  Hence,  the identity~\eqref{idenmain} holds for a generic choice of parameters. Note that $\frac{1}{d-1}\leq \eta(a)\leq 1.$ Hence, we expect that the diophantine exponent $w_p(\vec{a})$ to be dense in the interval $[1,2-2/d]$ as $q\to \infty$.  We give strong numerical evidence for this in Section~\ref{numeric}.
\subsection{Outline of the proofs}
We give an outline of the proof of  Theorem~\ref{main}. The proof is based on induction on $d.$ The base case $d=3$ was essentially  proved in the previous work of the third author  \cite[Theorem 1.10]{complexity}. 
Our algorithm starts with searching for the lattice points of $\mathcal{L}(\vec{a})$ inside a convex region defined by the intersection of two balls. There is a similar step in  the work of Ross and Selinger \cite{Selinger1}. Sarnak and Ori \cite{Parzanchevski} explained this step  in terms of Lenstra's work~\cite{Lenstra}.  If the convex region is defined by a system of linear inequalities in a fixed dimension then the general result of Lenstra \cite{Lenstra} implies this search is polynomially solvable. We use a variant of Lenstra's argument  that is developed in~\cite[Theorem 1.10]{complexity} and Conjecure~\ref{conj} to reduce the problem to dimension $d-1$.  At the final stage of our algorithm, we need to represent a given integers $m$ as a sum of two squares if it is possible. We apply Pollard's rho algorithm to factor $m$ into primes, and  check if all the prime factors  with the odd exponent are congruent to $1$ mod 4.  Finally, we use Schoof's algorithm~\cite{Schoof} to express each prime divisor  \(p \equiv 1~\mathrm{ mod }~ 4\) as a sum of two squares. An important feature of our algorithm is that it has been implemented for $d=4$~\cite{gitlab} and $d=3$ \cite{gitlab1}, and it runs and terminates quickly.
%
%
%The preceding two observations imply that
%. Define
%\(e(\epsilon) := h\log_q p\), so that
%\(e(\epsilon) = 2(1 + \epsilon) + O\left(\frac {\log \log q} {\log q}\right)\).
%

%
%
%
%\textbf{Problem 2:} Let \(q \equiv 1~\mathrm{ mod }~ 4\) be prime, and fix \(p\) and
%\(x \in P^2(\mathbb Z/q\mathbb Z)\) such that \(p\) is an odd prime not
%a quadratic residue\(~\mathrm{ mod }~ q\) (or, if it is, require \(a^2 + b^2 + c^2\)
%to also be a quadratic residue\(~\mathrm{ mod }~ q\), where \((a,b,c)\) is any
%representative of \(x\)). Find the smallest \(k \in \mathbb N\) such
%that problem (1) has a solution, where \(m = p^k\) and \((a,b,c)\) is
%allowed to be any representative of \(x\) such that
%\(a^2 + b^2 + c^2 \equiv p^k~\mathrm{ mod }~ q\).
%
%Denote the solution to problem (2) by \(h\). Section (2) shows that,
%assuming conjecture (1), there is always a solution to problem (1) if
%\(\sqrt {p^k}/(q|\vec u_i|) > O((\log q)^\gamma)\) for each
%\(i \in \{1,2,3\}\). (TODO: \(e \leq 4\))
%

\subsection*{Acknowledgements}\noindent
We thank Brandon
Boggess for his help for implementing  the code of Theorem~\ref{main}.
We also thank Professor Peter Selinger for publicly providing a very
useful Haskell package (newsynth) which was used in our code.

\section{Proof of Theorem~\ref{main} and \ref{algorithm}\label{mainthm}}
\subsection{$\delta$-LLL reduced basis} In this section we define a $\delta$-LLL reduced basis of $\mathbb{R}^d,$ and give a proof of Theorem~\ref{algorithm}.  We cite a theorem due to Babai on the shape of the LLL-reduced basis.  We refer the reader to \cite[Section 1]{LLL} for a detailed discussion of the LLL-algorithm. We first recall the Gram-Schmidt process.
\begin{defi}
Let $\vec{v}_1,\dots,\vec{v}_k$ be $k$ linearly independent  vectors in $\mathbb{R}^n.$ The Gram-Schmidt orthogonalization of  $\vec{v}_1,\dots,\vec{v}_k$
is defined inductively  by $\tilde{\vec{v}}_i=\vec{v}_i-\sum_{j=1}^{i-1}\mu_{i,j}\tilde{\vec{v}}_{j},$ where $\mu_{i,j}:=\frac{\langle \vec{v}_i,\tilde{\vec{v}}_j \rangle}{\langle \tilde{\vec{v}}_j,\tilde{\vec{v}}_j  \rangle}.$
\end{defi}
Next, we define a $\delta$-LLL reduced basis of $\mathbb{R}^d$ for any $1/4<\delta<1.$
\begin{defi}
A basis $\{\vec{v}_1,\dots,\vec{v}_d    \}\subset\mathbb{R}^d$ is a $\delta$-LLL reduced basis if the following holds:
\begin{enumerate}
\item  $|\mu_{i,j}|\leq 1/2$, for every $1 \leq  i \leq n$, and  $ j<i,$
\item  $\delta |\tilde{\vec{v}}_i|^2 \leq |\mu_{i+1,i} \tilde{\vec{v}}_i + \tilde{\vec{v}}_{i+1}|^2$  for for every $1 \leq  i < n.$ 
\end{enumerate}
\end{defi}
\begin{rem}\label{remlll}
By~\cite[Proposition 1.26]{LLL}, the LLL-algorithm transforms a given basis $B$ of a lattice $\mathcal L \subset \mathbb{Z}^d$ in $O(d^4\log(M(B)))$ operatins into a $\delta$-LLL reduced basis of $\mathbb{R}^d$, where $M(B)$ is defined in  \eqref{MB}.
\end{rem}
We cite the following theorem from \cite[Theorem 5.1]{Babai}.
\begin{thm}[Babai]\label{babai}
Let $\{\vec{v}_1,\dots,\vec{v}_d \}$ be a $\delta-$LLL reduced basis with $\delta=3/4.$ Let $\theta_k$ denote the angle between $\vec{v}_k$ and the linear subspace $U_k=\sum_{j\neq k}\mathbb{R}\vec{v}_j.$ Then, for every $1\leq k\leq d,$ 
\[ \sin \theta_k \geq (\frac{\sqrt{2}}{3})^d.\] 
\end{thm}
We give a proof of Theorem~\ref{algorithm}.
\begin{proof}
We give an LLL-reduced basis for the lattice $\mathcal L(\vec{a}).$ Assume the $a_0\not\equiv 0~\mathrm{mod}~q.$ Let $|\tilde{a}_0|\leq \frac{q-1}{2}$  be the integer such that  $\tilde{a}_0\equiv (a_0)^{-1}~\mathrm{mod}~q$.  Let $v_0:=(q,0,\dots,0)\in L(\vec{a}),$ and 
\[v_i:=(-\tilde{a}_0a_i, \delta_{1,i},\dots,\delta_{d-2,i}) \]
for $1\leq i \leq d-2,$ where $\delta_{i,j}=1$ if $i=j,$ and $\delta_{i,j}=0$ otherwise. Since the co-volume of $\mathcal L(\vec{a})$ is $q$, it follows that $\{v_0,\dots,v_{d-2}  \}$ is a $\mathbb{Z}$ basis for $\mathcal L(\vec{a}).$ We apply
 the LLL basis reduction
algorithm on $\{v_0,\dots,v_{d-2}  \}$ for $\delta=3/4$ and obtain a 3/4-LLL reduced basis 
\(B_L:=\{\vec u_0,\dots,\vec u_{d-2}\}\) for $\mathcal L(\vec{a})$ in $O(\log(q))$ steps; see Remark~\ref{remlll}. By \cite[Proposition~1.12]{LLL}, we have 
$$
\min_{B} M(B)\leq M(B_L)\leq 2^{(d-2)/2}\min_{B} M(B).
$$
 Hence, 
 $$
 0\leq \log_q(M(B_L))-\eta(\vec{a})\leq \frac{d-2}{2\log_2q}=O_d(1/\log q).
 $$
This concludes the proof of Theorem~\ref{algorithm}.
\end{proof}

\subsection{Proof of Theorem~\ref{main}}\label{section2}
Recall the notations while formulating Theorem~\ref{main}.  Let $\vec{a}=(a_0,\dots,a_{d}),$ where \(a_{d-1}\equiv a_d\equiv 0~\mathrm{mod}~q\). Assume that  \( \vec{s}:=(\frac{n_0}{p^h},\dots,\frac{n_d}{p^h})\in S^d(\mathbb{Z}[1/p]) \) is a minimal lift of $\vec{a},$ where $n_i\in \mathbb{Z}.$ Hence, we have    
\begin{equation*}
\begin{split}
&n_0^2+\dots+n_d^2=p^{2h}, 
\\
&n_i \equiv p^ha_i~\mathrm{ mod }~q \text{ for } 0\leq i\leq d.
\end{split}
\end{equation*} 
More generally, let  $N\leq q^A$ for some fixed $A>0$ be an integer,  and $b_i \in \mathbb{Z}$ for $ 0\leq i\leq d-2,$ where  $\sum_{i=0}^{d-2}b_i^2\equiv N~\mathrm{mod}~q$. Theorem~\ref{main} follows from the 
 following Proposition. 
 \begin{prop}\label{intlift}
 Assume $(*)$ and $d\geq 3$, we develop a polynomial-time algorithm in $\log(q)$ that finds a solution $(t_0,\dots,t_d)\in \mathbb{Z}^{d+1}$, if it exists,  to 
 \begin{equation}\label{ineq}(qt_0 + b_0)^2 + \dots + (qt_{d-2} + b_{d-2})^2 + (qt_{d-1})^2 + (qt_{d})^2 = N.\end{equation}
If there is no integral solution, it terminates in polynomial-time in $\log(q).$
 \end{prop}
\begin{proof}[Proof of Theorem~\ref{main}] For $0 \leq h \leq 4\log_{p} q,$ let $N=p^{2h}$ and $b_i\equiv a_i p^h~\mathrm{mod}~q$ for $0\leq i\leq d-2.$ By theorem~\cite[Theorem 1.2]{optimal} the diophantine equation 
\eqref{ineq} has a solution for every $(3+o_q(1))\log_{p} q \leq h \leq 4\log_{p} q.$ Our goal is to find the smallest $h$ such that the equation~\eqref{ineq} has a solution, and then find a solution  to the equation~\eqref{ineq}. For $0 \leq h \leq 4\log_{p} q,$   apply the algorithm in Proposition~\ref{intlift}, in order to find an integral solution to the 
 equation~\ref{ineq}.  If there exits such a solution $(t_{h,0},\dots,t_{h,d}),$ then 
 $$\vec{s}_h:= (\frac{qt_{h,0}+b_0}{p^h},\dots,\frac{qt_{h,d-2}+b_{d-2}}{p^h},\frac{qt_{h,d-1}}{p^h},\frac{qt_{h,d}}{p^h})$$
 is a  lift for $\vec{a}\in S^d(\mathbb{Z}[1/p]).$ Otherwise the algorithm in Proposition~\ref{intlift} terminates in polynomial-time in $\log(q)$ with no solutions,  and $\vec{a}$ does not have any integral lift $\vec{s}\in S^{d}(\mathbb{Z}[1/p])$ with $H(\vec{s})=p^h$. We have a lift $\vec{s}_h$ for every $(3+o_q(1))\log_{p} q \leq h \leq 4\log_{p} q,$ let $h_{min}$ be the smallest exponent $0\leq h \leq  4\log_{p} q$ such that the lift  $\vec{s}_{h}$ exists.  Then $\vec{s}_{h_{min}}$ is a minimal lift and   this concludes the proof of Theorem~\ref{main}.
  \end{proof}

Next, we prove two auxiliary lemmas and finally give a proof of Proposition~\ref{intlift}. By rearranging~\eqref{ineq}, we have
\begin{equation}\label{maineq}t_{d-1}^2 + t_d^2 = N/q^2 -(t_0 + b_0/q)^2 - \dots - (t_{d-2} + b_{d-2}/q)^2.
\end{equation}
Let $Q(\vec{t}):=N/q^2 - |\vec{t}+\frac{1}{q}\vec{b}|^2,$ where $\vec{t}= (t_0,\dots,t_{d-2}),$ $\vec{b}=(b_0,\dots,b_{d-2})$ and $|.|$ is the Euclidean norm.

 Recall the definition of $A_{Q,r}$ from~\eqref{Afr}, where $r>0$ is some real number. By Conjecture~\ref{conj}, if $ |A_{Q,r}|> C_{\gamma}(\log N)^{\gamma}$ then the equation~\eqref{maineq} has a solution, where $\vec{t}\in A_{Q,r}.$ 
%First, we give a parametrization of $(t_0,t_1,t_2)\in \mathbb{Z}^2$, where $Q(t_0,t_1)\in \mathbb{Z}.$
%Let $k:=\frac{N-a_0^2-a_1^2}{4q}$.  Since $a_0^2 +a_1^2 \equiv N  \text{  mod } 4q$, $k\in\mathbb{Z}$. 
Let \(k := (N - |\vec{b}|^2)/q\). Since $|\vec{b}|^2\equiv N~\mathrm{ mod }~q$, $k\in \mathbb{Z}.$ We can further rearrange~\eqref{maineq}:
\[t_{d-1}^2 + t_d^2 = (k - 2\langle \vec{b}, \vec{t}\rangle)/q -|\vec{t}|^2.\]
Note that $\vec{t}\in A_{Q,r}$ iff  the following two conditions are satisfied:

\begin{itemize}
\item
  {Condition 1:}
  \(|\vec{t}+\frac{1}{q}\vec{b}|^2 \leq N/q^2\), and $|\vec{t}|<r.$
\item
  {Condition 2:} \(2\langle \vec{b}, \vec{t}\rangle \equiv k~\mathrm{ mod }~ q\).
\end{itemize}
%We define a triple \((t_1,t_2,t_3)\) to be a
%candidate if it satisfies the above conditions. 
%The strategy is to check if
%\(m/q^2 - (qt_1 + a)^2 - (qt_2 + b)^2 - (qt_3 + c)^2\) is a sum of two
%squares for each candidate in such a way that conjecture (1) guarantees
%the search will terminate in \(O(poly(\log q))\) time. We do this by
%parameterizing the set of candidates in such a way that the expression
%in question becomes a quadratic function satisfying the hypothesis of
%lemma (1).
We first focus on Condition 2.  Without loss of generality, we assume that  $a_0\not\equiv 0$ mod $q$.  Then $2b_0\equiv 2p^ha_0 \not\equiv 0~\mathrm{mod}~q $ and $2b_0$  has an inverse mod $q$. Let $\tilde{b}_0\leq \frac{q-1}{2}$ be the integer such that  $\tilde{b}_0\equiv (2b_0)^{-1}~\mathrm{mod}~q$. Then
$\vec{t}_0:=(k\tilde{b}_0,0,\dots,0)$  is a solution for the congruence equation in Condition (2). Since $p^h\vec{a}\equiv \vec{b}~\mathrm{mod}~q$,
the integral solutions of Condition (2) are the translation of the lattice points of $\mathcal L(\vec{a})$ by $\vec{t}_0.$ Let \(\{\vec u_0,\dots,\vec u_{d-2}\}\) be the 3/4-LLL reduced basis 
 for $\mathcal L(\vec{a})$ that is defined in the proof of Theorem~\ref{algorithm}. We write 
\[\vec{t}_0+\frac{1}{q}\vec{b}=\sum_{i=0}^{d-2}c_i \vec{u}_i,\]
for sum $c_i\in \frac{1}{q^2}\mathbb{Z}.$ Let  $\tilde{\vec{t}}_0= \sum_{i=0}^{d-2}r_i \vec{u}_i,$ where $|r_i|\leq 1/2$ and $c_i-r_i\in \mathbb{Z}$ for every $0\leq i\leq d-2.$
% Since $d$ is fixed, by preforming the Shortest Lattice Vector algorithm, we can find a shortest vector  $\tilde{\vec{t}}_0$  satisfying Condition (2) in polynomial-time. We parametrize the solutions to Condition 2 as
%\[\vec t = \tilde{\vec{t}}_0 + \sum_{i=0}^{d-2}x_i\vec{u}_i,\] where $x_i \in \mathbb{Z}.$
Assume that $\vec{t}\in \mathbb{Z}^{d-1}$  satisfies Condition (2). Then, there exists  a one to one correspondence between $\vec{t}$ and 
\(\vec{x} := (x_0,\dots,x_{d-2})\in \mathbb{Z}^{d-1}\), such that:
\[ 
\tilde{\vec{t}}_0+\sum_{i=0}^{d-2}x_i\vec{u}_i=\vec{t}+\frac{1}{q}\vec{b}.
\]
Let
\[F(\vec{x}) :=N/q^2 - | \tilde{\vec{t}}_0+\sum_{i=0}^{d-2}x_i\vec{u}_i|^2.\] Note that $F(\vec{x})=Q(\vec{t})$ by the above correspondence, and  $F(\vec{x})\in \mathbb{Z}$ for every $\vec{x}\in \mathbb{Z}^{d-1}.$
Clearly Condition (1) is satisfied if and only if \(F \geq 0\).  

%Next, we list  all the integral vectors $\vec{x}\in\mathbb{Z}^{d-1}$ such that $F(\vec{x})$ is positive. 

We prove two general lemmas for listing the positive values of $F(\vec{x})$. Assume that $\{\vec{w}_1, \dots,\vec{w}_m\}$ is a $3/4$-LLL basis for $\mathbb{R}^m.$ Let $\tilde{\vec{w}}_0=\sum_{i=1}^m s_i\vec{w}_i,$ where $|s_i|<1/2.$ Define
$$
H(x_1,\dots,x_m):=M^2-|\tilde{\vec{w}}_0+\sum_{i=1}^{m}x_i\vec{w}_i|^2,
$$
where $M$ is some real number. 

\begin{lem}\label{line}
Assume that $  \alpha |\vec{w}_k|> M $ for some $\alpha>0$, and $H(x_1,\dots,x_m)>0$, then $$|x_k|\leq \alpha(\frac{3}{\sqrt{2}})^m+1.$$
\end{lem} 
\begin{proof} Since  $H(x_1,\dots,x_m)>0,$ we have 
$$
0<H(x_1,\dots,x_m)\leq M^2-\big(  \sin \theta_k(x_k+s_k)|\vec{w}_k|\big)^2. 
$$
By Theroem~\ref{babai} and $ \alpha  |\vec{w}_k|>  M,$ we have 
$$
\big|  \sin \theta_k (x_k+s_k)|\vec{w}_k| \big|\geq  (\frac{\sqrt{2}}{3})^m  (|x_k|-1/2)\frac{M}{\alpha} .$$
Hence,
$$
|x_k|\leq \alpha (\frac{3}{\sqrt{2}})^m+1/2.
$$
This concludes the lemma. 
\end{proof}

\begin{lem}\label{boxlem}
Assume that $2m|\vec{w}_i|< M$ for $1\leq i\leq m$. Let $A_i:=\frac{M}{m|\vec{w}_i|}-1/2$  and $ C:= \prod_{i=1}^{m} [-A_i,A_i].$  Then  $H(\vec{x})$ is positive for every $\vec{x}\in C$ and negative outside $2m(\frac{3}{\sqrt{2}})^mC$.
\end{lem}
 \begin{proof}
  Recall that $H(x_1,\dots,x_m)=M^2-|\tilde{\vec{w}}_0+\sum_{i=1}^{m}x_i\vec{w}_i|^2,$ and $\tilde{\vec{w}}_0=\sum_{i=1}^m s_i\vec{w}_i,$ where $|s_i|<1/2.$ Assume that $\vec{x}\in C.$
 By the triangle inequality 
\begin{align*}
H(\vec{x})\geq M^2 - \big(\sum_{i=1}^{m} (|x_i|+1/2) |\vec{w}_i|\big)^2 
&\geq M^2 - \big(\sum_{i=1}^{m} (A_i+1/2) |\vec{w}_i|\big)^2
\\&\geq M^2 -M^2 =0.
\end{align*}
%Since $|u_0|<|u_2|$ then $|(t_0,t_1)|\leq |x| |u_1|+(1+|y|)|u_2|.$ Let $A$ , $B$ and $C$ be as in~\eqref{box}. 
%Then for every $(x,y)\in [-A,A]\times[-B,B]$, we have 
%$$
%|x| |u_1|+(1+|y|)|u_2|\leq (\sqrt{N}/q) -|u_1| < (\sqrt{N}/q) -1.
%$$
%Hence, $F(x,y)>0$ if $(x,y)\in [-A,A]\times[-B,B]$. 
Next, we show that $H$ is negative outside $2m(\frac{3}{\sqrt{2}})^mC.$ Assume that $\vec{y}:=(y_1,\dots,y_{m})\not\in 2m(\frac{3}{\sqrt{2}})^mC.$ Hence, there exits $1\leq k\leq m$ such that $|y_k| \geq 2m(\frac{3}{\sqrt{2}})^mA_k.$  By Theorem~\ref{babai} and the assumption $2m|\vec{w}_k|< M$, we obtain
\begin{align*}
F(\vec{y})&\leq M^2-\big(\sin \theta_k (y_k+r_k)|\vec{w}_k|\big)^2 
\\
&\leq M^2-\Big((\frac{\sqrt{2}}{3})^m \big(2m(\frac{3}{\sqrt{2}})^m(\frac{M}{m|\vec{w}_k|}-1/2)-1/2 \big)|\vec{w}_k|\Big)^2<0
 \end{align*}
%The above  inequality implies that if  $x \geq 10A$, then 
% $$|(t_0,t_1)|=|u_0+xu_1+yu_2| > \sqrt{N}/q+ \big(3\sqrt{N}/q|u_1|-10\big)|u_1|/2.$$ 
%We assume that $\frac{\sqrt{N}}{q|u_2|}\geq 14/3$ and $1<|u_1|<|u_2|$, then $|(t_0,t_1)|> \sqrt{N}/q+1 $ and hence $F(x,y)$ is negative. Similarly, if  $y \geq 10 B $ then
%  $$|(t_0,t_1)|=|u_0+xu_1+yu_2| > \sqrt{N}/q+ \big(3\sqrt{N}/(2q|u_2|)-6\big)|u_2|.$$ 
% Since $\frac{\sqrt{N}}{q|u_2|}\geq 14/3$, it follows that  $|(t_0,t_1)|> \sqrt{N}/q+1.$ Hence,  $F(x,y)$ is negative. Therefore, if $(x,y)\notin 10 \times  C$, then $F(x,y)$ is negative.  
% 
 This concludes our lemma. 
\end{proof}
Finally, we give a proof of Proposition~\ref{intlift}.
\begin{proof}[Proof of Proposition~\ref{intlift}] 
 Recall the notations and the assumptions while formulating Proposition~\ref{intlift}. We develop an algorithm that finds a solution to the equation~\eqref{maineq} in polynomial-time in $\log(q),$ and if it does not have a solution, it terminates in polynomial-time in $\log(q)$. 
 
 First,  assume that 
$2(d-1)|\vec{u}_i|< \frac{N}{\sqrt{q}}$ for every $1\leq i\leq d-1.$ By Lemma~\ref{boxlem}, there exists a box $C$ such that $F(\vec{x})$ is positive inside $C$ and it is 
negative outside  $2(d-1)(\frac{3}{\sqrt{2}})^{d-1}C$. We consider two cases. 
\begin{itemize}
\item{Case 1:} if $|C|\leq C_{\gamma}\log(N)^{\gamma},$  
\item{Case 2:} if $|C|> C_{\gamma}\log(N)^{\gamma}.$
\end{itemize}
where $|C|=\prod_{i=0}^{d-2}A_i,$  $C_{\gamma}$ and $\gamma$ are defined in  Conjecture~\ref{conj}.

For Case 1, we check if any point  $\vec{x}\in 2(d-1)(\frac{3}{\sqrt{2}})^{d-1}C$ gives a solution to the equation~\ref{maineq} as follows.  We factor $F(\vec{x})$ in polynomial-time in $\log(q)$ into its prime powers,  by our assumed polynomial-time factoring algorithm. We check if all the prime factors  with the odd exponent are congruent to $1$ mod 4.  Finally, we use Schoof's algorithm~\cite{Schoof} to express each prime divisor  \(p \equiv 1~\mathrm{ mod }~ 4\) as a sum of two squares. Since  $|C|\leq C_{\gamma}\log(N)^{\gamma},$ this conduces the proof of Proposition~\ref{intlift}. 

For Case 2, by Conjecture~\ref{conj},  there exists $\vec{x}\in C$ such that 
$F(\vec{x})=t_{d-1}^2+t_d^2$ for some $t_{d-1},t_d\in \mathbb{Z},$  where $|\vec{x}|\leq C_{\gamma}\log(N)^{\gamma}. $ Similarly, we find such  $|\vec{x}|\leq C_{\gamma}\log(N)^{\gamma}$ in polynomial time. This conduces the proof of Proposition~\ref{intlift} if $2(d-1)|\vec{u}_i|< \frac{N}{\sqrt{q}}$ for every $1\leq i\leq d-1.$

 Otherwise, there exists $0\leq k\leq d-2$
such that $2(d-1)|\vec{u}_k|> \frac{N}{\sqrt{q}}.$ By Lemma~\ref{line}, $$|x_k|\leq  2(d-1)(\frac{3}{\sqrt{2}})^{d-1}+1.$$
Since $d$ is fixed, there are only a bounded number of choices for $x_k\in \mathbb{Z}.$ Let $x_k=l$ for some $l\in\mathbb{Z},$ where $|l|\leq 2(d-1)(\frac{3}{\sqrt{2}})^{d-1}+1.$ Hence,
\[F(\vec{x}) :=N/q^2 - | \tilde{\vec{t}}_0+l \vec{u}_k+\sum_{i\neq k}^{d-2}x_i\vec{u}_i|^2.\]
We write uniquely  $\tilde{\vec{t}}_0+l \vec{u}_k= \vec{u}_{k,1}+ \vec{u}_{k,2},$ where $\vec{u}_{k,1}=\sum_{i\neq k} \alpha_i \vec{u}_i$ and $\vec{u}_{k,2}$ is orthogonal to $\sum_{i\neq k}\mathbb{R}\vec{u}_i.$ Hence,
\[
F(\vec{x})=M- | \tilde{\vec{w}}_0+\sum_{i\neq k}^{d-2}y_i\vec{u}_i|^2
\]
where $M:=(N/q^2 - |\vec{u}_{k,2}|^2),$ $\tilde{\vec{w}}_0:=\sum_{i\neq k} s_i \vec{u}_i,$ where $|s_i|\leq 1/2$ and $s_i-\alpha_i\in \mathbb{Z},$ and $y_i=x_i+\alpha_i-s_i.$ Let 
\[
G_{l,k}(\vec{y}):=M- | \tilde{\vec{w}}_0+\sum_{i\neq k}^{d-2}y_i\vec{u}_i|^2.
\]
Next, we use a similar argument as in the beginning of our proof.  We assume that $2(d-1)|\vec{u}_i|< M$ for all $i\neq k,$ and proceed with the same argument on $G_{l,k}(\vec{y})$ as $F(\vec{x})$.  
We either find a solution for the equation~\eqref{maineq}, or find another variable with bounded value. Since the dimension $d$ is bounded this algorithm terminates in polynomial time in $\log(q).$ This completes the proof of Proposition~\ref{intlift}. 
\end{proof}
Finally, we give a proof of Theorem~\ref{dioexpthm}
\begin{proof}
Assume that 
\[h \geq (\log_p q)\Big(1 + \eta(\vec{a}) + \frac{2d+ \log_2 2C_{\gamma}+\gamma\log_2 5\log q}{\log_2(q)}\Big).\]
 Let \(B_L:=\{\vec u_0,\dots,\vec u_{d-2}\}\) be the LLL-reduced basis that is introduced in 
the proof of Theorem~\ref{algorithm}. It follows from the proof of   Theorem~\ref{algorithm} that 
$$
\eta(\vec{a})+\frac{d}{\log_2(q)}\geq \log_q(M(B_L)).
$$
Hence, for every $0 \leq i\leq d-2$, we have 
\[
h\geq (\log_p q) \Big(1+\log_q(|\vec{u}_i|) +\frac{d+ \log_2 2C_{\gamma}+\gamma\log_25\log q}{\log_2(q)}\Big).
\]
Let $N:=p^{2h}$, we have 
\[
\frac{\sqrt{N}}{q}\geq 2d |\vec{u}_i|C_\gamma (\log q^5)^{\gamma}.
\]
Assume that $N\leq q^5,$ then 
$$
\frac{\sqrt{N}}{q}\geq 2d |\vec{u}_i|C_\gamma (\log N)^{\gamma}.
$$
By the proof of Proposition~\ref{intlift}, if follows that there exists an integral lift $\vec{s}\in S^{d}(\mathbb{Z}[1/p])$ with  $H(\vec{s})=p^h.$ Therefore,
$$
w_p(\vec{a}) \leq \frac{d-1}{d}(1 + \eta(\vec{a})) + O_d(\log\log(q)/\log(q)).
$$
This concludes the proof of Theorem~\ref{dioexpthm}.
\end{proof}

\section{Numerical results}\label{numeric}

We now give numerical evidence for Conjecture~\ref{conj} by testing identity~\ref{expiden} for $d = 4$. Figure~\ref{linear}, shown in the introduction, was produced by choosing the three non-zero coordinates in $S^4(\mathbb{Z}/q\mathbb{Z})$ randomly on a logarithmic scale. This was done specifically by first choosing an integer $r$ randomly from $60$ to $125$ for each coordinate, then choosing an integral representative of the coordinate randomly from $0$ to $10^{-r}q$. This was done $100$ times for each of eight $130$-digit primes listed below, and all points were included in the figure:\\

\begin{itemize}
\item{$q_1 = $} 8 7 1 4 7 7 2 9 7 6 3 8 1 6 9 7 8 5 9 5 5 7 9 6 5 3 5 5 7 9 1 7 4 5 9 9 7 9 8 7 7 2 8 3 0 0 2 2 6 8 1 2 7 3 2 5 1 9 7 0 2 2 5 0 4 6 3 3 2 9 2 6 5 8 2 6 0 3 3 2 6 8 9 2 6 1 7 1 6 1 0 3 9 2 2 1 1 7 5 0 0 4 4 9 5 7 7 1 9 8 9 6 3 2 0 7 7 6 7 9 4 5 5 5 1 4 6 5 8 1
\item{$q_2 = $} 4 4 8 6 8 4 7 1 8 8 0 5 2 9 1 9 9 9 8 8 4 0 4 4 5 6 3 3 1 9 7 1 8 5 8 3 7 4 1 2 6 9 4 5 5 5 9 2 0 4 7 0 5 0 5 2 7 2 9 6 0 8 8 0 2 9 2 1 1 8 6 9 9 5 7 4 8 6 0 4 1 9 6 0 2 0 2 5 0 7 5 6 3 9 1 9 7 6 3 9 9 7 3 6 2 3 2 4 1 1 7 4 5 5 4 6 4 0 2 3 9 9 4 0 4 6 8 4 9 1
\item{$q_3 = $} 3 0 0 7 4 8 1 5 1 9 7 7 4 8 4 1 7 0 0 4 4 7 1 6 9 1 0 3 6 9 9 4 8 0 4 7 9 2 5 7 0 7 5 1 5 8 0 0 4 4 2 9 1 7 2 5 0 5 4 9 3 1 6 9 6 2 4 2 1 1 7 3 8 1 6 5 6 2 6 2 0 4 2 2 1 6 5 3 1 1 5 5 1 8 1 7 1 8 9 5 1 2 4 2 5 4 9 5 5 0 50 1 9 2 4 7 6 6 4 7 3 9 3 2 4 0 5 4 9
\item{$q_4 = $} 1 1 8 2 9 1 1 1 4 1 4 9 7 5 4 1 1 4 0 5 8 6 6 1 5 3 5 9 2 3 0 2 4 3 5 9 2 8 3 4 25 9 9 3 5 9 5 1 4 5 7 9 2 3 8 9 0 1 0 0 9 0 1 1 4 3 7 2 4 7 0 2 2 9 5 7 5 0 5 4 3 1 8 0 7 5 7 8 8 6 0 0 3 7 6 2 2 5 0  8 2 2 0 1 4 7 7 0 4 3 2 4 0 6 7 8 5 9 7 5 2 2 0 7 1 2 5 5 1
\item{$q_5 = $} 5 4 5 2 2 1 2 1 5 1 8 4 4 2 6 4 4 5 3 1 1 5 5 2 1 7 7 0 3 6 2 7 1 2 8 1 5 3 0 0 4 5 6 7 8 8 0 7 3 8 7 0 2 6 0 6 3 7 1 7 2 0 0 6 4 1 4 9 8 7 4 7 9 9 1 4 1 5 0 8 3 1 8 2 1 2 0 2 2 5 9 8 6 2 0 9 1 3 7 3 7 4 1 7 3 8 5 1 1 5 7 9 6 2 9 0 4 0 7 3 2 9 0 9 1 9 4 8 8 3
\item{$q_6 = $} 6 5 3 9 0 1 0 9 3 5 8 3 6 2 4 8 6 9 8 1 3 0 7 9 7 8 9 5 0 0 9 7 0 2 3 6 5 5 0 2 8 9 4 5 0 6 1 0 9 6 0 0 3 3 5 0 0 3 5 4 9 6 4 6 8 0 7 8 4 8 5 9 7 7 7 0 3 8 7 6 0 7 6 1 2 7 6 7 2 8 5 5 5 8 4 4 0 8 8 9 4 4 4 1 3 2 8 2 4 6 9 6 9 1 4 1 1 3 6 0 3 5 6 9 7 1 3 1 5 2
\item{$q_7 = $} 1 9 6 6 6 0 5 1 7 0 3 5 1 8 8 3 0 0 3 6 6 9 3 8 1 3 2 6 7 2 6 2 5 4 8 6 5 0 9 5 1 8 5 7 3 6 1 5 9 6 1 0 5 1 1 2 9 0 3 1 7 2 3 8 3 1 5 2 3 2 4 2 1 3 3 5 3 4 0 7 8 7 9 5 3 2 2 3 3 7 4 9 5 0 9 7 1 8 5 0 2 8 7 5 1 7 5 6 1 7 2 5 1 8 3 5 2 5 3 4 9 9 0 1 7 2 9 7 9 2
\item{$q_8 = $} 4 5 7 8 4 8 7 2 7 4 2 0 8 4 8 2 7 2 2 0 6 7 2 3 8 0 3 9 3 4 0 3 6 9 5 8 1 5 9 7 5 4 2 8 5 9 6 1 0 6 2 8 5 6 8 5 2 5 9 8 4 9 4 4 9 5 3 6 4 9 1 7 8 9 3 7 3 5 1 2 6 5 6 3 8 9 9 8 8 5 6 7 6 4 6 3 9 7 8 5 6 8 6 8 2 9 3 8 0 2 0 2 6 3 7 9 1 3 3 2 9 7 3 9 2 3 5 3 2 8
\end{itemize}

\subsection{Generic Coordinates}
There are several cases which are worthy of special consideration. The generic element of $S^2(\mathbb{Z}/q\mathbb{Z})$ has coordinates of size $q$, so we expect $\eta(\vec{a}) = 1/3$ and $w_p(\vec{a}) = 1$ for most lattices. Figure~\ref{generic} shows that this is indeed the case, using the same primes and number of points as Figure~\ref{linear}. The coordinates are chosen between $0$ and $q$ on a linear, rather than logarithmic scale. The horizontal lines observed on the small-scale are a result of $H(\vec{a})$, and therefore $w_p(\vec{a})$, taking much more discrete values than $\eta(\vec{a})$.

\begin{figure}[H]
\centering
\raisebox{-0.6cm}{
\includegraphics[width=\textwidth]{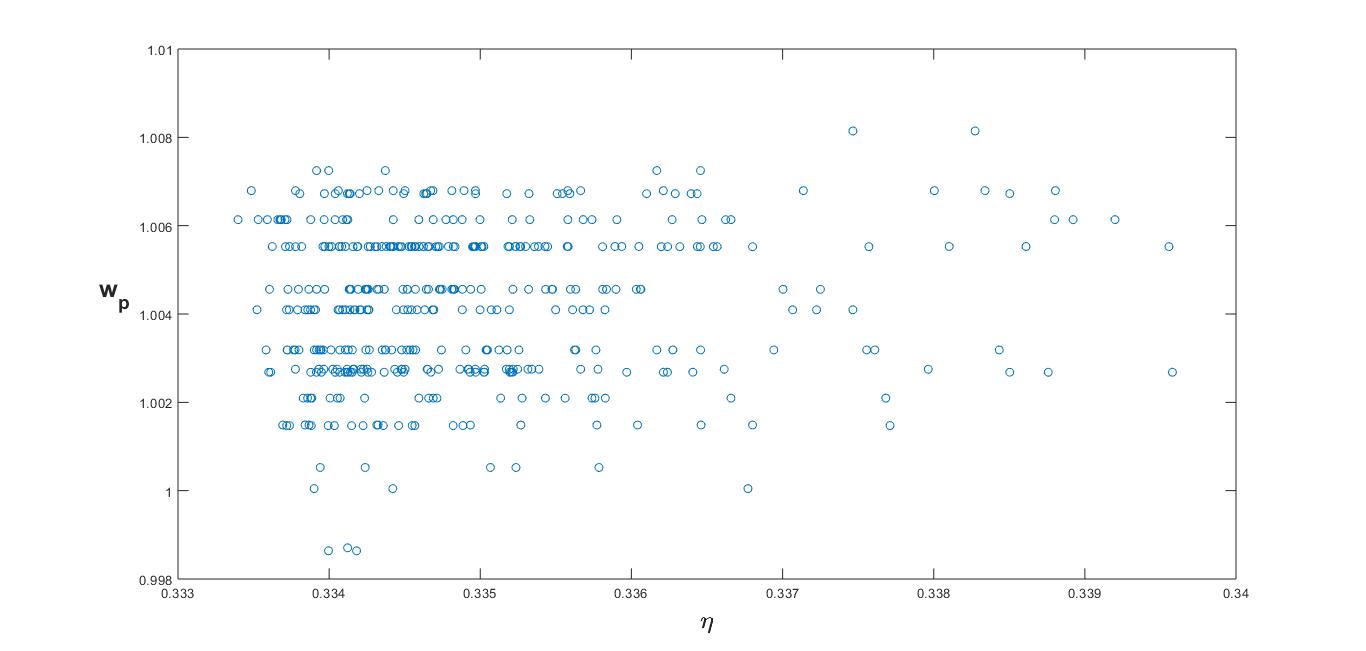}
}
\caption{Generic Coordinates}
\label{generic}
\end{figure}

\subsection{Small Coordinates}
When all coordinates are small, the lattice is quite high in the cusp, and therefore one expects $\eta(\vec{a}) = 1$ and $w_p(\vec{a}) = 3/2$, which is observed in Figure~\ref{small}. Here all coordinates are chosen between $0$ and $10^{-125}q$.

\begin{figure}[H]
\centering
\raisebox{-0.6cm}{
\includegraphics[width=\textwidth]{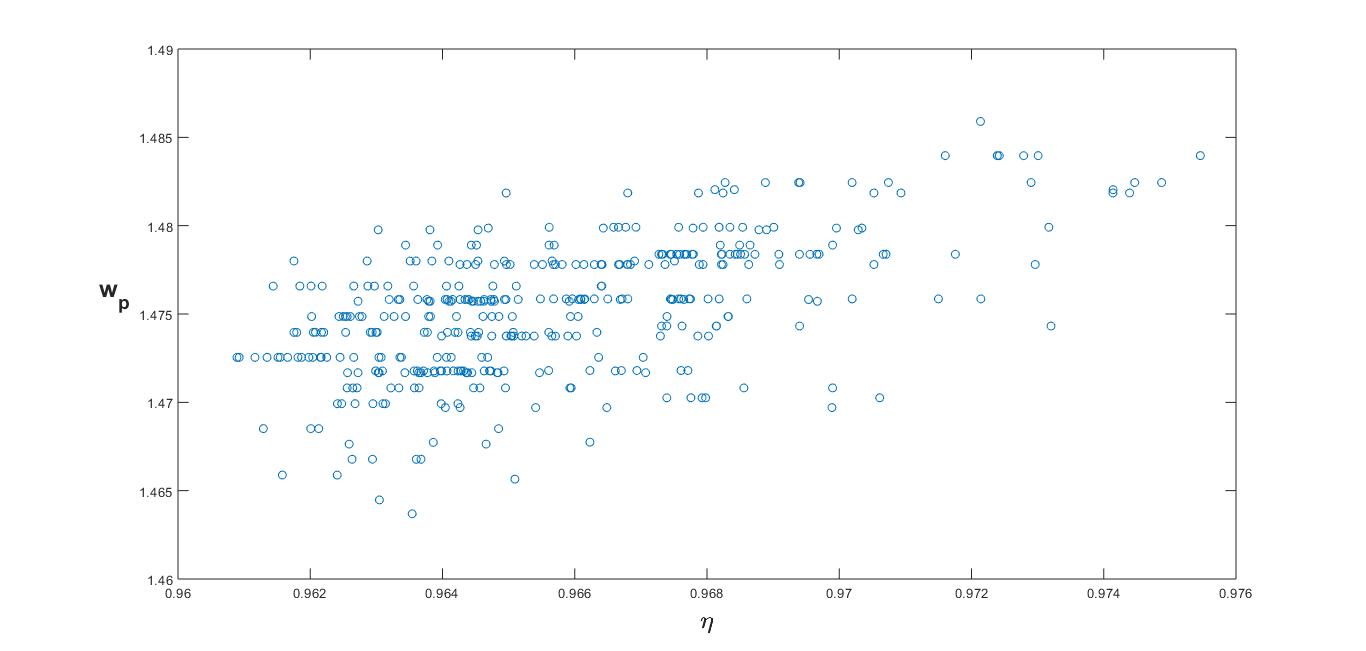}
}
\caption{Small Coordinates}
\label{small}
\end{figure}

\subsection{Other Cusp Regions}
One can explore additional cusp cases by fixing one or two coordinates and varying the rest on a logarithmic scale. Figures ~\ref{oneFixed} and ~\ref{twoFixed} show that identity~\ref{expiden} still holds in these two cases. The fixed coordinate is set to $1$, and the other coordinates are chosen as in Figure~\ref{linear}. Note that in Figure~\ref{twoFixed}, where only one coordinate is large, the lattices are relatively high in the cusp, but the corresponding points still adhere to the theoretical line.

\begin{figure}[H]
\centering
\raisebox{-0.6cm}{
\includegraphics[width=\textwidth]{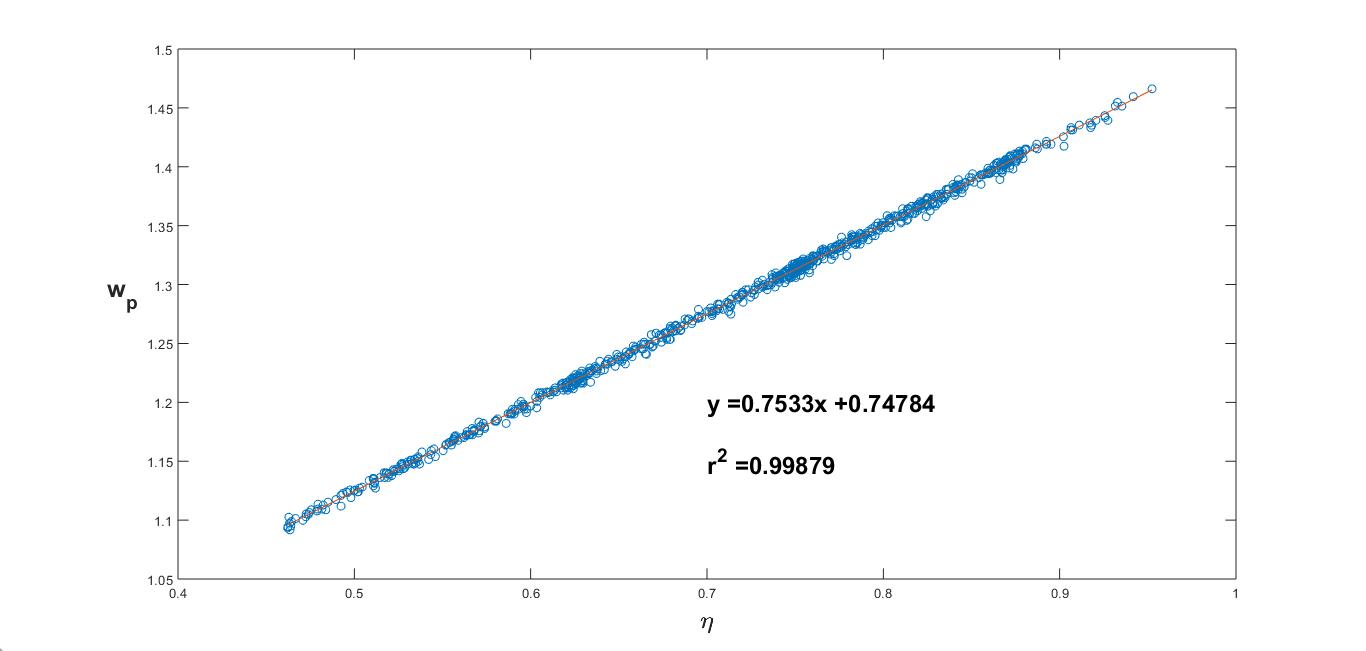}
}
\caption{One Coordinate Fixed}
\label{oneFixed}
\end{figure}

\begin{figure}[H]
\centering
\raisebox{-0.6cm}{
\includegraphics[width=\textwidth]{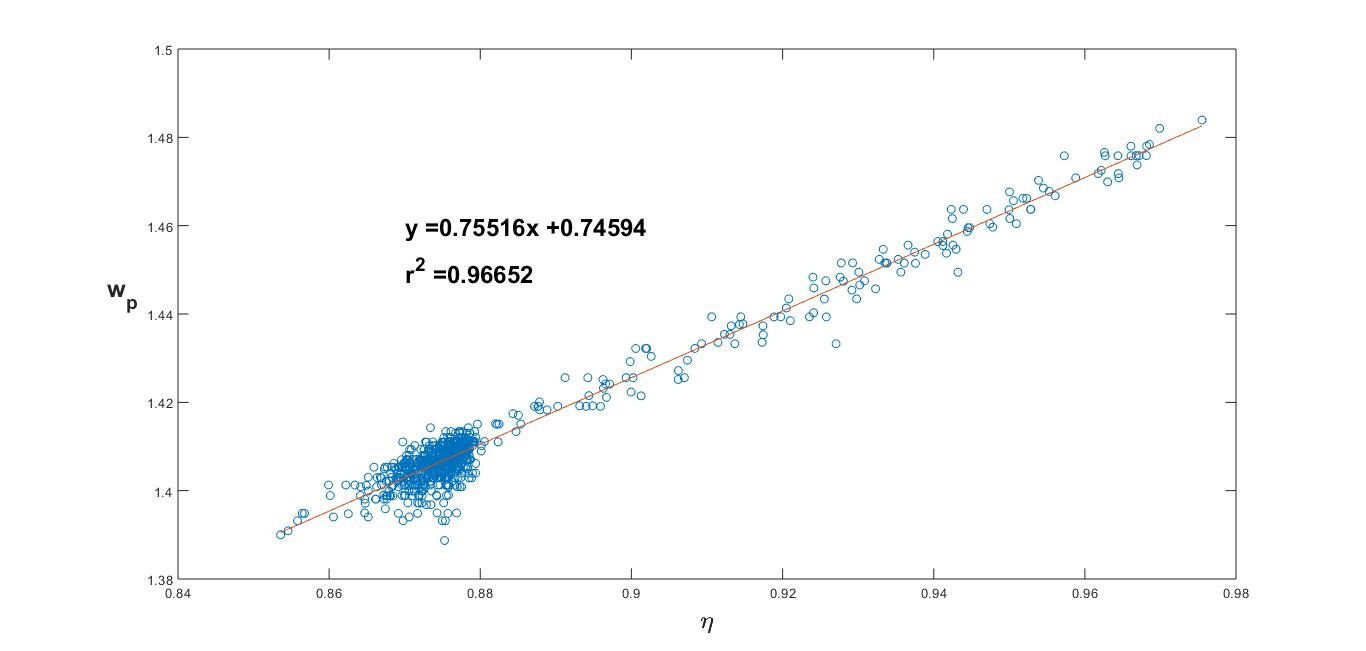}
}
\caption{Two Coordinates Fixed}
\label{twoFixed}
\end{figure}

\bibliographystyle{alpha}
\bibliography{Paper}

\newcommand{\etalchar}[1]{$^{#1}$}
\begin{thebibliography}{HMS{\etalchar{+}}18}

\bibitem[Bab86]{Babai}
L.~Babai.
\newblock On {L}ov\'{a}sz' lattice reduction and the nearest lattice point
  problem.
\newblock {\em Combinatorica}, 6(1):1--13, 1986.

\bibitem[BKS17]{BKS}
T.D. Browning, V.~Vinay Kumaraswamy, and R.S. Steiner.
\newblock Twisted linnik implies optimal covering exponent for $s^3$.
\newblock {\em International Mathematics Research Notices}, page rnx116, 2017.

\bibitem[HMS{\etalchar{+}}18]{gitlab}
M.~W. Hassan, Y.~Mao, N.~T. Sardari, R.~Smith, and X.~Zhu.
\newblock {\em 5 Squares Algorithm}, August 2018.
\newblock
  \url{https://gitlab.com/5-Squares-Algorithm/diophantine-approximation}.

\bibitem[Klo27]{Kloos}
H.~D. Kloosterman.
\newblock On the representation of numbers in the form {$ax^2+by^2+cz^2+dt^2$}.
\newblock {\em Acta Math.}, 49(3-4):407--464, 1927.

\bibitem[Len83]{Lenstra}
H.~W. Lenstra, Jr.
\newblock Integer programming with a fixed number of variables.
\newblock {\em Math. Oper. Res.}, 8(4):538--548, 1983.

\bibitem[LLL82]{LLL}
A.~K. Lenstra, H.~W. Lenstra, Jr., and L.~Lov\'{a}sz.
\newblock Factoring polynomials with rational coefficients.
\newblock {\em Math. Ann.}, 261(4):515--534, 1982.

\bibitem[LPS88]{LPS}
A.~Lubotzky, R.~Phillips, and P.~Sarnak.
\newblock Ramanujan graphs.
\newblock {\em Combinatorica}, 8(3):261--277, 1988.

\bibitem[Mar88]{Margulis}
G.~A. Margulis.
\newblock Explicit group-theoretic constructions of combinatorial schemes and
  their applications in the construction of expanders and concentrators.
\newblock {\em Problemy Peredachi Informatsii}, 24(1):51--60, 1988.

\bibitem[PLQ08]{Petit2008}
Christophe Petit, Kristin Lauter, and Jean-Jacques Quisquater.
\newblock {\em Full Cryptanalysis of LPS and Morgenstern Hash Functions}, pages
  263--277.
\newblock Springer Berlin Heidelberg, Berlin, Heidelberg, 2008.

\bibitem[PS18]{Parzanchevski}
Ori Parzanchevski and Peter Sarnak.
\newblock Super-golden-gates for {$PU(2)$}.
\newblock {\em Adv. Math.}, 327:869--901, 2018.

\bibitem[RS16]{Selinger1}
Neil~J. Ross and Peter Selinger.
\newblock Optimal ancilla-free {${\rm Clifford}+T$} approximation of
  {$z$}-rotations.
\newblock {\em Quantum Inf. Comput.}, 16(11-12):901--953, 2016.

\bibitem[{Sar}15a]{optimal}
N.~T {Sardari}.
\newblock {Optimal strong approximation for quadratic forms}.
\newblock {\em ArXiv e-prints}, October 2015.

\bibitem[Sar15b]{SarnakGate}
Peter Sarnak.
\newblock {\em {Letter to Scott Aaronson and Andy Pollington on the
  Solovay-Kitaev Theorem}}, February 2015.
\newblock
  \url{https://publications.ias.edu/sarnak/paper/2637}{https://publications.ias.edu/sarnak/paper/2637}.

\bibitem[{Sar}17a]{complexity}
N.~T {Sardari}.
\newblock {Complexity of strong approximation on the sphere}.
\newblock {\em ArXiv e-prints}, March 2017.

\bibitem[Sar17b]{gitlab1}
N.~T. Sardari.
\newblock {\em Navigating LPS Ramanujan Graphs}, March 2017.
\newblock \url{https://gitlab.com/ntalebiz/navigating-lps-ramanujan-graphs}.

\bibitem[Sar18]{Naser}
Naser~T. Sardari.
\newblock Diameter of ramanujan graphs and random cayley graphs.
\newblock {\em Combinatorica}, Aug 2018.

\bibitem[Sch85]{Schoof}
Ren\'{e} Schoof.
\newblock Elliptic curves over finite fields and the computation of square
  roots mod {$p$}.
\newblock {\em Math. Comp.}, 44(170):483--494, 1985.

\end{thebibliography}

\end{document}